\documentclass[a4paper,12pt]{amsart}

\usepackage{amsfonts, amsmath, amssymb, mathrsfs}
\usepackage{amsthm, url}
\usepackage{graphicx,epstopdf}
\usepackage{mathdots}
\usepackage{hyperref}
\usepackage{paralist}

\theoremstyle{plain}
\newtheorem{Theorem}{Thm}[section]

\newtheorem{Lem}[Theorem]{Lemma}
\newtheorem{Cor}[Theorem]{Corollary}

\newtheorem*{Thm*}{Theorem}
\newtheorem*{Lem*}{Lemma}
\newtheorem*{Rem*}{Remark}

\theoremstyle{definition}

\newtheorem{Exm}[Theorem]{Example}

\newtheorem{theorem}{Theorem}

\theoremstyle{definition}

\newcommand{\R}{\mathbb{R}}

\newcommand{\N}{\mathbb{N}}
\newcommand{\Z}{\mathbb{Z}}
\renewcommand\P{\mathbb{P}}
\newcommand{\C}{\mathbb{C}}

\DeclareMathOperator{\PGL}{PGL}
\DeclareMathOperator{\GL}{GL}
\DeclareMathOperator{\Aut}{Aut}

\DeclareMathOperator{\conv}{conv}
\DeclareMathOperator{\diag}{diag}

\renewcommand{\rho}{\varrho}
\renewcommand{\#}[1]{|#1|}

\renewcommand{\>}{\rangle}

\title[Classification of toric projective varieties]{Classification of
  toric projective varieties up to projective automorphisms}

\author{Friedrich Knop} \address{Friedrich Knop: Department
  Mathematik, FAU Erlangen-N\"urnberg, Cauerstra\ss e 11, 91058
  Erlangen, Germany} \email{friedrich.knop@fau.de}

\author{Rainer Sinn} \address{Rainer Sinn: Fachbereich Mathematik und
  Informatik, FU Berlin, Arnimallee 2, 14195 Berlin, Germany}
\email{rsinn@zedat.fu-berlin.de}

\subjclass[2010]{Primary: 14M25, 52B20}

\begin{document}

\begin{abstract}
  Toric subvarieties of projective space are classified up to
  projective automorphisms.
\end{abstract}

\maketitle

\section{The result}

\noindent A closed subvariety $X\subseteq\P^N=\P^N(\C)$ will be called
\emph{toric} if it is the orbit closure of a subtorus
$T\subseteq\PGL(N+1)$. The aim of this note is to classify toric
subvarieties of $\P^N$ as abstract subvarieties, i.e., without $T$
being part of the structure. More precisely, we give a description of
all toric subvarieties up to projective equivalence where two
subvarieties $X,X'\subseteq\P^N$ will be called \emph{projectively
  equivalent} if there is $g\in\PGL(N+1)$ (possibly not commuting with
$T$) such that $g(X)=X'$.

Toric subvarieties can be constructed as follows: For any $d\in\N$, we
choose some non-empty subset $S\subseteq\Z^d$ with $\#S\le N+1$. Let
$S=\{m_0,\ldots,m_r\}$ (where $r\le N$, by assumption). Then
$X(S)=\overline{Tx}$ is a toric subvariety of $\P^N$ where $T$ is the image of
\begin{equation}\label{eq:hom}
(\C^*)^d\to\PGL(N+1):t\mapsto\diag(t^{m_0},\ldots,t^{m_r},1,\ldots,1)
\end{equation}
(with $N-r$ ones) and $x:=[1:\ldots:1:0:\ldots:0]$ (with $N-r$
zeros). Observe that $X(S)$ depends on the choice of an enumeration of
$S$ but clearly any two choices lead to projective equivalent
subvarieties.

We call a subset $S\subseteq\Z^d$ \emph{affinely generating}
if $\Z^d$ is generated as a group by the differences $m-m'$ with
$m,m'\in S$. Two subsets $S,S'\subseteq\Z^d$ are \emph{affinely
  equivalent} if $S'=\Phi(S)$ where $\Phi$ is an affine transformation
of the form $\Phi(m)=Am+b$ with $A\in \GL(d,\Z)$ and $b\in\Z^d$.

Our main result is the following:

\begin{theorem}\label{thm:main}

  For any $d,N\in\N$ with $d\le N$ the map $S\mapsto X(S)$ induces a
  bijection between

  \begin{itemize}

    \item affinely generating subsets $S\subseteq\Z^d$ with
    $\#S\le N+1$ up to affine equivalence and

  \item toric subvarieties $X\subseteq\P^N$ of dimension $d$ up to
    projective equivalence.
   
  \end{itemize}
  
\end{theorem}

Our main application is to toric varieties attached to lattice
polytopes. Recall that a \emph{lattice polytope} $P$ is the convex
hull inside $\R^d$ of a finite subset of $\Z^d$. Thus,
$S_P:=P\cap\Z^d$ is finite and $P=\conv S_P$. Lattice polytopes are in one-to-one correspondence with pairs of complete (abstract) toric varieties together with a complete and ample linear series.
We call a lattice polytope 
\emph{$\Z$-solid} if $S_P$ affinely generates
$\Z^d$. Clearly every $\Z$-solid lattice polytope is solid, i.e.,
$\dim P=d$. The converse is not true. Counterexamples are provided at
the end of the paper (Example~\ref{exm:solid}). Every lattice polytope
with the integer decomposition property is $\Z$-solid and in particular,
every lattice polygon is $\Z$-solid.

For a lattice polytope $P$, put $X(P):=X(S_P)$. Then applying
Theorem~\ref{thm:main} to $X(P)$ we get:

\begin{Cor}

  Let $P,P'\subseteq\R^d$ be two $\Z$-solid lattice polytopes with
  $\#{S_P},\#{S_{P'}}\le N+1$. Then $X(P)$ and $X(P')$ are projectively
  equivalent if and only if $P$ and $P'$ are affinely equivalent.
  In particular, it follows that $X(P) = X(P')$.

\end{Cor}

We cannot simply drop the assumption that $P$ and $P'$ are $\Z$-solid.
\begin{Exm}
Consider the two full-dimensional tetrahedra 
\[
P = \conv\{0,e_1,e_2,e_3\} \text{ and } P' = \conv\{0,e_1,e_2,e_1+e_2+2e_3\}
\]
in $\R^3$. Clearly, the lattice points in $P$ generate the $\Z^3\subset\R^3$ as a group. On the other hand, the lattice polytope $P'$ is not $\Z$-solid. The toric varieties corresponding to these two polytopes with $4$ lattice points is $\P^3$ in both cases. In particular, they are projectively equivalent, whereas the lattice polytopes are not affinely equivalent.
\end{Exm}

One can also wonder when two toric varieties of the form $X(S)$ are
projectively equivalent if $S$ is not necessarily affinely
generating. In this case, let $M(S)\subseteq\Z^d$ be the subgroup
generated by all $m-m'$, $m,m'\in S$. Choose any isomorphism
$\Phi:M(S)\to\Z^e$ and let $\tilde S:=\Phi(S-m)$ where $m$ is any
element of $S$. Then $\tilde S$ affinely generates $\Z^e$.

\begin{Cor}

  Let $S\subseteq\Z^d$ be a subset with $\#S\le N+1$. Then $X(S)$ is
  projectively equivalent to $X(\tilde S)$.
  
\end{Cor}

\begin{Rem*}
  In this paper, we work for convenience over $\C$ only. It is not
  difficult to see, that all results hold over an arbitrary
  algebraically closed ground field. In positive characteristic, only
  the proof of Lemma~\ref{lem:dim} requires a slight modification.
\end{Rem*}

\section{The proof}
As a general reference for toric varieties from the algebraic geometry point of view, we refer to \cite{CLS}. The basic facts from representation theory that we use throughout the paper can be found in \cite{Hump}.

The proof of Theorem~\ref{thm:main} proceeds in several steps. First
we observe that the map $S\mapsto X(S)$ is well defined:

\begin{Lem}

  If $S,S'\subseteq\Z^d$ are affinely equivalent then
  $X(S),X(S')\subseteq\P^N$ are equal, so in particular
  projectively equivalent.
  
\end{Lem}

\begin{proof}
  We show that any affine linear isomorphism $m\mapsto Am+b$ of $\Z^d$ does 
  not change $Tx\subset\P^N$. Indeed, $A$ results simply in a 
  reparametrization of $(\C^*)^d$ while $b$ corresponds to a 
  multiplication by a non-zero constant factor which drops out by
  homogeneity.
\end{proof}

Next we show that dimensions are preserved:

\begin{Lem}\label{lem:dim}
  Let $S\subseteq\Z^d$ be affinely generating. Then
  $\chi:(\C^*)^d\to X(S):t\mapsto tx$ is an open embedding. Hence
  $\dim X(S)=d$.
\end{Lem}

\begin{proof}
  The map $\chi$ is dominant by definition. Hence the image is
  open. It suffices to show that $\chi^{-1}(x)=1$. We have
  \begin{align*}
    \chi^{-1}(x)=&\{t\in(\C^*)^d\mid t^{m_0}=\ldots=t^{m_r}\}=\\
    =&\{t\in(\C^*)^d\mid t^{m_1-m_0}=\ldots=t^{m_r-m_0}=1\}.
  \end{align*}
The latter group is trivial since $\Z^d$ is affinely generated by $S$, so that these characters separate the torus.
\end{proof}

To show surjectivity, we first investigate arbitrary homomorphisms
$\rho:T\to\PGL(N+1)$, where $T$ is a torus. The image is contained in a maximal torus of $\PGL(N+1)$, which is conjugate to the maximal torus $\{\diag(1,x)\colon x\in(\C^*)^N\}\subset\PGL(N+1)$. This torus can clearly be lifted to $\GL(N+1)$ and hence also $\rho$ can be lifted to a homomorphism
$\tilde\rho:T\to\GL(N+1)$. Any two lifts
differ by a homomorphism of $T$ into the scalars. 
Via this lift $\tilde\rho$, the torus $T$ acts on $\C^n$, which decomposes into simultaneous eigenspaces and the eigenvalues corresponding to an eigenspace are given by the corresponding eigencharacter on $T$. 
Thus the set of
eigencharacters $E(\rho;\P^N)\subseteq M_T$ is well defined up to translation by an element of the lattice of characters $M_T$.

\begin{Lem}
  For every toric subvariety $X\subseteq\P^N$ of dimension $d$ there
  is a affinely generating subset $S\subseteq\Z^d$ such that that $X$
  is projectively equivalent to $X(S)$. 
\end{Lem}

\begin{proof}

  Let $X=\overline{Tx}$ with $T\subset\PGL(N+1)$ a torus. Lift the
  inclusion $\rho$ to a homomorphism $\tilde\rho:T\to\GL(N+1)$. Let
  $E:=E(\rho;\P^N)$ be the set of its eigencharacters and let
  $\C^{N+1}=\bigoplus_{\chi\in E}V_\chi$ be the corresponding
  decomposition in eigenspaces. Lift also $x$ to a vector
  $v\in\C^{N+1}$ and decompose it: $v=\sum_{\chi\in E}v_\chi$. Put
  $S_0:=\{\chi\in E\mid v_\chi\ne0\}$ and choose an eigenbasis whose
  first $r+1$ members are the $v_\chi$ with $\chi\in S_0$. After a
  projective transformation we may assume that this eigenbasis is the
  standard basis. If $S_0=\{\chi_0,\ldots,\chi_r\}$ we see that then
  \[
    Tx=\{[\chi_0(t):\ldots:\chi_r(t):0:\ldots:0]\in\P^N\mid t\in T\}
  \]
  Let $M\subseteq M_T$ be the subgroup generated by all ratios
  $\chi_i\chi_j^{-1}$ and let $\chi_i':=\chi_i\chi_0^{-1}$. Then $M$
  is affinely generated by $S:=\{\chi_i'\mid 0\le i\le j\}$. Fix an
  isomorphism $M\cong\Z^e$ for some $e\in\N$ such that $S$ becomes an
  affinely generating subset of $\Z^e$. By construction, we have
  $X=\overline{Tx}=X(S)$. Lemma~\ref{lem:dim} then implies $e=d$.
\end{proof}

Next we prove the key lemma. For any subvariety $X\subseteq\P^N$ let
$G(X)\subseteq\Aut X$ be the group of all automorphisms which can be
extended to an automorphism of $\P^N$. If $X=\overline{Tx}$ is
now toric, then restriction to $X$ yields a homomorphism $\rho:T\to G(X)$.

\begin{Lem}\label{lem:max}
  $T_0:=\rho(T)$ is a maximal torus of $G(\overline{Tx})$. 
\end{Lem}

\begin{proof}
  Set $X = \overline{Tx}$. Choose a maximal torus $T_1\subseteq G(X)$ containing $T_0$. We are
  going to show that every $t_1\in T_1$ lies in $T_0$. Since $t_1$
  normalizes $T_0$, it permutes the $T_0$-orbits of $X$. The open
  $T_0$-orbit $Tx$ is unique, hence it is preserved by $t_1$. Thus,
  there is $t_0\in T_0$ with $t_1x=t_0x$, implying that
  $t':=t_1t_0^{-1}$ fixes $x$. Since $t'$ centralizes $T_0$, it will
  fix every point of $Tx$. Thus it acts trivially on the closure
  $X$. This shows $t'=1$ and therefore $t_1=t_0\in T_0$.
  \end{proof}

The following Lemma finishes the proof of Theorem~\ref{thm:main}.

\begin{Lem}
  Let $S,S'\subseteq\Z^d$ be affinely generating and assume that
  $X:=X(S)$, $X':=X(S')$ are projectively equivalent. Then $S$, $S'$
  are affinely equivalent.
\end{Lem}

  \begin{proof}
    Let $\<X\>\subseteq\P^N$ be the smallest linear subspace
    containing $X$. Then the linear independence of characters implies
    that $\<X\>=\P^r\subseteq\P^N$ where $r+1=\#S$.

    By assumption there is $g\in\PGL(N+1)$ with $g(X)=X'$.  Since then
    $g\<X\>=\<X'\>$ we deduce that $\dim \<X\>=r=\dim \<X'\>$ and that $g$
    is an automorphism of $\P^r$. Thus we may replace $\P^N$ by $\P^r$
    and hence assume that $r=N$. Since, in this case, the only element
    of $\PGL(N+1)$ which acts as identity on $X$ is the identity, we
    see that $G(X)$ (and likewise $G(X')$) is a subgroup of
    $\PGL(N+1)$.

    Let $T_0$ and $T_0'$ be the image of $T=(\C^*)^d$ in $\PGL(N+1)$
    corresponding to $S$ and $S'$, respectively, via the map defined 
    in \eqref{eq:hom}. Then, by
    Lemma~\ref{lem:max}, both $T_0$ and $T_1:=g^{-1}T_0'g$ are maximal
    tori of $G(X)$. It follows that there is $h\in G(X)$ with
    $h^{-1}T_1h=T_0$. Replacing $g$ by $gh$ we may assume additionally
    to $g(X)=X'$ that $T_0=T_1$.

    Lemma~\ref{lem:dim} also implies that $T\to T_0$ and $T\to T_0'$ are
    isomorphisms. Thus we get $\phi\in\Aut T$ such that the following
    diagram commutes:
    \[
      \begin{matrix}
        T&\overset{\rho_S}\to&T_0\\
        \llap{$\phi$}\downarrow&&\downarrow\rlap{${\rm Ad}(g)$}\\
        T&\overset{\rho_{S'}}\to&T_0'
      \end{matrix}
    \]
    Write $\Phi$ for the automorphism of $\Z^d$ corresponding to $\phi$, so that using $\rho_S\circ\phi^{-1}=\rho_{\Phi(S)}$ we get
    \[
      \rho_{S'}(t)=g\rho_{\Phi(S)}g^{-1}
    \]
    This means that the sets of eigencharacters $S'$ and $\Phi(S)$ are equal up to a
    translation, i.e., there is $b\in\Z^d$ such that $S'=\Phi(S)+b$.
\end{proof}

% As a corollary of the proof, one gets:

% \begin{Cor}

%   Let $X\subseteq\P^N$ be a toric subvariety which is not contained in
%   any proper linear subspace. Then there is a torus
%   $T\subseteq\PGL(N+1)$ acting effectively with an open orbit on
%   $X$. Moreover, any two such tori a conjugate by an element of
%   $N(X)\subseteq\PGL(N+1)$.
  
% \end{Cor}
  
We conclude the paper with an example of a solid but not $\Z$-solid
lattice polytope.

\begin{Exm}\label{exm:solid}

  Consider the lattice $M=\Z^d+\Z\delta\subseteq\R^d$ where
  $\delta=(\frac12,\ldots,\frac12)$. Then the simplex
  $P:=\conv\{0,e_1,\ldots,e_d\}$ is a lattice polytope. For $d\ge3$ we
  have $\delta\not\in P$. Hence $P$ is not $\Z$-solid for $d\ge3$. For
  $d\le2$ one can show that all solid lattice polytopes are
  $\Z$-solid.
  
\end{Exm}

% \bibliographystyle{abbrv}
%  \bibliography{lit}

\end{document}